\documentclass[leqno, 12pt]{article}
\usepackage{amsmath,amsfonts,amsthm,amssymb,indentfirst}
\usepackage{xy} \xyoption{all}
\usepackage[colorlinks]{hyperref}
\hypersetup{linkcolor=blue, urlcolor=blue, citecolor=red}

\newtheorem{theorem}{Theorem}
\newtheorem{lemma}[theorem]{Lemma}
\newtheorem{definition}[theorem]{Definition}
\newtheorem{corollary}[theorem]{Corollary}
\newtheorem{proposition}[theorem]{Proposition}

\theoremstyle{definition}
\newtheorem{example}[theorem]{Example}

\newcommand{\End}{\mathrm{End}}

\newcommand{\rad}{\mathrm{rad}}

\newcommand{\cl}{\mathrm{cl}}

\newcommand{\tnil}{\mathrm{TNil}}

\newcommand{\Z}{\mathbb{Z}}

\newcommand{\M}{\mathbb{M}}

\newcommand{\B}{\mathcal{B}}

\newcommand{\U}{\mathcal{U}}

\begin{document}

\title{Infinite-Dimensional Triangularizable Algebras}

\author{Zachary Mesyan}

\maketitle

\begin{abstract}
Let $\End_k(V)$ denote the ring of all linear transformations of an arbitrary $k$-vector space $V$ over a field $k$. We define $X \subseteq \End_k(V)$ to be \emph{triangularizable} if $V$ has a well-ordered basis such that $X$ sends each vector in that basis to the subspace spanned by basis vectors no greater than it. We then show that an arbitrary subset of $\End_k(V)$ is \emph{strictly} triangularizable (defined in the obvious way) if and only if it is topologically nilpotent. This generalizes the theorem of Levitzki that every nilpotent semigroup of matrices is triangularizable. We also give a description of the triangularizable subalgebras of $\End_k(V)$, which generalizes a theorem of McCoy classifying triangularizable algebras of matrices over algebraically closed fields.

\medskip

\noindent
\emph{Keywords:} triangular matrix, linear transformation, simultaneous triangularization, triangularizable algebra, nilpotent semigroup, endomorphism ring, function topology

\noindent
\emph{2010 MSC numbers:} 15A04, 16S50 (primary); 16W80, 47D03 (secondary)
\end{abstract}


\section{Introduction}

Given a field $k$ and a $k$-vector space $V$, we denote by $\End_k(V)$ the $k$-algebra of all linear transformations of $V$. We define a transformation $T\in \End_k(V)$ to be \emph{triangularizable} if $V$ has a well-ordered basis $(\B, \leq)$ such that $T$ sends each vector $v \in \B$ to the subspace spanned by $\{u \in \B \mid u \leq v\}$. If $V$ is finite-dimensional, then this notion is clearly equivalent to $T$ being representable as an upper-triangular matrix with respect to some basis for $V$, which in turn is equivalent to $T$ being representable as a lower-triangular matrix with respect to some basis for $V$. The above definition of triangularizability was first introduced in~\cite{ZM}, where various equivalent characterizations of transformations of this sort are given, along with other facts about them. 

In this paper we focus on subsets of $\End_k(V)$ that are (simultaneously) triangularizable, that is, ones where all elements are triangularizable with respect to a common well-ordered basis for $V$. We give a general tool for showing sets of transformations to be triangularizable (Lemma~\ref{tri-lemma}), which is analogous to the so-called ``Triangularization Lemma", proved for bounded linear operators on Banach spaces by Radjavi and Rosenthal~\cite{RR}. Using this tool we show that a subset $X$ of $\End_k(V)$ is triangularizable if and only if there exists a well-ordered (by inclusion) set of $X$-invariant subspaces of $V$, which is maximal as a well-ordered set of subspaces of $V$ (Proposition~\ref{subspaces-prop}).

Defining a transformation $T\in \End_k(V)$ to be \emph{strictly} triangularizable if $V$ has a well-ordered basis $(\B, \leq)$ such that $T$ sends each vector $v \in \B$ to the subspace spanned by $\{u \in \B \mid u < v\}$, we then characterize the strictly (simultaneously) triangularizable subsets of $\End_k(V)$ in Theorem~\ref{strictly-upper}. Specifically, $X \subseteq \End_k(V)$ is strictly triangularizable if and only if $X$ is topologically nilpotent (i.e., given any infinite list $T_1, T_2, T_3, \ldots \in X$ and any finite-dimensional subspace $W$ of $V$, one has $T_n \cdots T_2T_1(W)=0$ for some positive integer $n$). As immediate consequences of this theorem we obtain the well-known result of Levitzki~\cite{Levitzki} that every nilpotent multiplicative semigroup of matrices is triangularizable, and the statement that the nonunital $k$-algebra generated by a topologically nilpotent subset of $\End_k(V)$ is also topologically nilpotent (Corollary~\ref{nilp-alg-cor}). Additionally, this theorem generalizes a recent result of Chebotar~\cite{Chebotar}, who used our approach to triangularization to answer a question of Greenfeld, Smoktunowicz, and Ziembowski. Theorem~\ref{strictly-upper} can also be viewed as an analogue of the result~\cite[Corollary 11]{Turovskii} of Turovskii that a quasinilpotent semigroup (i.e., one where the spectrum of each of its operators is contained in $\{0\}$) of compact operators on a Banach space is chain-triangularizable (that is, possesses a maximal chain of invariant subspaces).

Most of the paper, starting with Section~\ref{rad-sect}, is devoted to describing triangularizable subalgebras of $\End_k(V)$. In Theorem~\ref{main-theorem} we show that a subalgebra $R$ of $\End_k(V)$ is triangularizable if and only if $R$ is contained in a $k$-subalgebra $A$ of $\End_k(V)$ such that $A/\rad(A) \cong k^{\Omega}$ as topological $k$-algebras for some set $\, \Omega$, and $\, \rad (A)$ is topologically nilpotent. We then give another equivalent condition to these in the case where the field $k$ is algebraically closed (Corollary~\ref{alg-closed-cor}). A consequence of Theorem~\ref{main-theorem} is that a subalgebra $R$ of a matrix ring $\M_n(k)$ is triangularizable if and only if $R/\rad(R) \cong k^m$ as $k$-algebras, for some positive integer $m$, and moreover, if $k$ is algebraically closed, then $R$ being triangularizable is also equivalent to $R/\rad(R)$ being commutative (Corollary~\ref{McCoy-cor}). This generalizes a well-known result of McCoy~\cite{McCoy}. Various examples illustrating our results are given along the way.

\section{Preliminaries}

All rings will be assumed to be unital, unless stated otherwise. We denote the set of the integers by $\Z$, and the set of the positive integers by $\Z^+$. Also, given a ring $R$ and $n \in \Z^+$ we denote by $\M_n(R)$ the ring of all $n\times n$ matrices over $R$.

In the following subsections we review the standard topology on $\End_k(V)$, some basics of the theory of topological rings, and a few facts about triangularizable elements of $\End_k(V)$.

\subsection{Topological Rings} \label{top-sect}

Let $X$ and $Y$ be sets, and let $Y^X$ denote the set of all functions $X \to Y$.
The \emph{function} (or \emph{finite}, or \emph{pointwise}) \emph{topology} on $Y^X$ has a base of open sets of the following form: $$\{f \in Y^X \mid f(x_1) = y_1, \dots, f(x_n) = y_n\} \ (x_1, \dots, x_n \in X, y_1, \dots, y_n \in Y).$$ It is straightforward to see that this coincides with the product topology on $Y^X = \prod_X Y$, where each component set $Y$ is given the discrete topology.  As a product of discrete spaces, this space is Hausdorff.

Now let $V$ be a vector space over a field $k$. Then $\End_k(V) \subseteq V^V$ inherits a topology from the function topology on $V^V$, which we shall also call the \emph{function topology}. Under this topology $\End_k(V)$ is a topological ring (see, e.g.,~\cite[Theorem~29.1]{Warner}), i.e., a ring $R$ equipped with a topology that makes $+:R\times R \to R$, $-:R \to R$, and $\cdot :R\times R \to R$ continuous. Alternatively, we may describe the function topology on $\End_k(V)$ as the topology having a base of open sets of the following form: $$\{S \in \End_k(V) \mid S|_W=T|_W\} \ (T \in \End_k(V), W \subseteq V \text{ a finite-dimensional subspace}),$$ where $S|_W$ and $T|_W$ denote the restrictions of $S$ and $T$, respectively, to $W$. Observe that when $V$ is finite-dimensional, $\End_k(V)$ is discrete in this topology.

Let us also review some standard facts about topological rings that will be needed later on. Given any topological ring $R$ and an ideal $I$ of $R$, the quotient $R/I$ may be viewed as a topological ring in the \emph{quotient topology}, where a subset of $R/I$ is open if and only if its preimage under the canonical projection $R \to R/I$ is open in $R$. (See, e.g.,~\cite[Theorem 5.4]{Warner} for details.) Also, two topological rings $R_1$ and $R_2$ are \emph{isomorphic as topological rings} if there is a \emph{topological isomorphism} $R_1 \to R_2$, i.e., a function that is both a ring isomorphism and a homeomorphism. The concept \emph{isomorphic as topological $k$-algebras} is defined analogously. 

Finally recall that a map $\phi : X \to Y$ of topological spaces is \emph{open} if $\phi (U)$ is open whenever $U$ is an open subset of $X$, and it is \emph{closed} if $\phi (U)$ is closed whenever $U$ is a closed subset of $X$.

\subsection{Triangular Transformations}

Recall that a binary relation $\leq$ on a set $X$ is a \emph{partial order} if it is reflexive, antisymmetric, and transitive. If, in addition, $x\leq y$ or $y \leq x$ for all $x,y \in X$, then $\leq$ is a \emph{total order}. If $\leq$  is a total order and, moreover, every nonempty subset of $X$ has a least element with respect to $\leq$, then $\leq$ is a \emph{well order}.

Given a subset $X$ of a vector space, we denote by $\langle X \rangle$ the subspace spanned by $X$. 

\begin{definition} \label{tri-def}
Let $k$ be a field, $\, V$ a $k$-vector space, $\B$ a basis for $\, V$, and $\, \leq$ a partial ordering on $\B$. 

We say that $T\in \End_k(V)$ is \emph{triangular with respect to} $(\B, \leq)$ if $T(v) \in \langle \{u \in \B \mid u \leq v\} \rangle$ for all $v \in \B$, and that $T$ is \emph{strictly triangular with respect to} $(\B, \leq)$ if $T(v) \in \langle \{u \in \B \mid u < v\} \rangle$ for all $v \in \B$.

We say that $X \subseteq \End_k(V)$ is \emph{triangular with respect to} $(\B, \leq)$, respectively \emph{strictly triangular with respect to} $(\B, \leq)$, if each $T \in X$ is triangular with respect to $\, (\B, \leq)$, respectively, strictly triangular with respect to $\, (\B, \leq)$.

If $T\in \End_k(V)$ $($or $X \subseteq \End_k(V)$$)$ is triangular, respectively strictly triangular, with respect to some \emph{well-ordered} basis for $\, V$, then we say that $T$ $($or $X$$)$ is \emph{triangularizable}, respectively \emph{strictly triangularizable}.
\end{definition}

Sometimes we shall find it more convenient to index bases with ordered sets rather than ordering the bases themselves, when dealing with triangularization. 

It is easy to see that in the case where $V$ is a finite-dimensional $k$-vector space, $T \in \End_k(V)$ is triangularizable in the above sense if and only if there is a basis for $V$ with respect to which $T$ is upper-triangular as a matrix, if and only if there is a basis for $V$ with respect to which $T$ is lower-triangular as a matrix. A more detailed discussion of the choices made in the above definition can be found in~\cite[Section 3]{ZM}, and various characterizations of triangularizable transformations are given in~\cite[Theorem 8]{ZM}.

Next we recall a couple of results from~\cite{ZM} that will be used frequently. Given $k$-vector spaces $W \subseteq V$ and a transformation $T \in \End_k(V)$, we say that $W$ is $T$-\emph{invariant} if $T(W) \subseteq W$.

\begin{lemma}[Part of Proposition 16 in \cite{ZM}] \label{inverse-prop}
Let $k$ be a field, $V$ a $k$-vector space, and $T \in \End_k(V)$ triangular with respect to some well-ordered basis $\, (\B, \leq)$ for $\, V$. Also for each $v \in \B$ let  $\pi_v \in \End_k(V)$ be the projection onto $\, \langle v \rangle$ with kernel $\, \langle \B \setminus \{v\} \rangle$. Then $T$ is invertible if and only if $\pi_vT\pi_v \neq 0$ for all $v \in \B$.  
\end{lemma}

\begin{definition}
Let $k$ be a field, $V$ a $k$-vector space, and $X \subseteq \End_k(V)$. We say that $X$ is \emph{topologically nilpotent} if the sequence $\, (T_i \cdots T_2T_1)_{i=1}^\infty$ converges to $\, 0$ in the function topology on $\, \End_k(V)$, for every infinite list $T_1, T_2, T_3, \ldots \in X$ of not necessarily distinct transformations. Moreover, if $\, \{T\}$ is topologically nilpotent for some $T \in \End_k(V)$, then we shall refer to $T$ as \emph{topologically nilpotent}.
\end{definition}

\begin{proposition}[Proposition 20 in \cite{ZM}] \label{top-nilpt-prop}
Let $k$ be a field and $\, V$ a nonzero $k$-vector space. The following are equivalent for any $T \in \End_k(V)$.
\begin{enumerate}
\item[$(1)$] $T$ is topologically nilpotent.
\item[$(2)$] $V = \bigcup_{i=1}^\infty \ker (T^i)$.
\item[$(3)$] $T$ is strictly triangularizable.
\item[$(4)$] $T$ is triangularizable, and if $\, (\B, \leq)$ is a well-ordered basis for $\, V$ with respect to which $T$ is triangular, then $T$ is strictly triangular with respect to $\, (\B, \leq)$.
\item[$(5)$] $T$ is triangularizable, and for all $a \in k$, we have $\, \ker(T-a \cdot 1) \neq \{0\}$ if and only if $a=0$.
\end{enumerate}
\end{proposition}

We conclude this section with a basic fact about triangularizable sets.

\begin{lemma} \label{closure-lemma}
Let $k$ be a field, $\, V$ a nonzero $k$-vector space, $\, (\B, \leq)$ a partially ordered basis for $\, V$, $X \subseteq \End_k(V)$, and $\, \cl(X)$ the closure of $X$ in the function topology. Then $X$ is triangular with respect to $\B$ if and only if $\, \cl(X)$ is.
\end{lemma}

\begin{proof}
Suppose that $X$ is triangular with respect to $\B$. If $X = \emptyset$, then $\cl(X) = \emptyset$. So we may assume that $X \neq \emptyset$. Letting $T \in \cl(X)$ and $v \in \B$ be arbitrary, there is some $S \in X$ such that $T|_{\langle v\rangle} = S|_{\langle v\rangle}$, by properties of the function topology. Hence $$T(v) = S(v) \in \langle \{u \in \B \mid u \leq v\} \rangle.$$ Since $v \in \B$ was arbitrary, it follows that $T$ is triangular with respect to $\B$. Since $T \in \cl(X)$ was arbitrary, we conclude that $\cl(X)$ is triangular with respect to $\B$.

Conversely, if $\cl(X)$ is triangular with respect to $\B$, then the same holds for any  subset of $\cl(X)$, and for $X$ in particular.
\end{proof}

\section{Triangularization Lemma}

Let $k$ be a field, $V$ a $k$-vector space, $X \subseteq \End_k(V)$, and $W\subseteq V$ an $X$-invariant subspace (i.e., $X(W) \subseteq W$). Then we shall denote by $\overline{X}$ the subset of $\End_k(V/W)$ consisting of the transformations $\overline{T}$ defined by $\overline{T}(v+W)=T(v)+W$, for $T \in X$. (It is routine to check that any such $\overline{T}$ is well-defined and $k$-linear.) We make the next definition following Radjavi and Rosenthal~\cite{RR}.

\begin{definition} \label{quotient-def}
A property $P$ on sets of $k$-vector space transformations is \emph{inherited by quotients} if given any $k$-vector spaces $\, W\subseteq V$, and any $X \subseteq \End_k(V)$ such that $X(W) \subseteq W$ and $X$ has property $P$, then $\overline{X} \subseteq \End_k(V/W)$ also has property $P$.
\end{definition}

The following is an analogue of the Triangularization Lemma~\cite[Lemma 1]{RR}, proved for bounded linear operators on Banach spaces by Radjavi and Rosenthal. It will be our main tool for showing sets of transformations to be triangularizable. The (short) proof is essentially the same as that of~\cite[Theorem 15]{ZM} (which says that any finite commutative subset of $\End_k(V)$ consisting of triangularizable transformations is triangular with respect to a common well-ordered basis), but we give the details here for completeness.

\begin{lemma}\label{tri-lemma}
Let $k$ be a field, and let $P$ be a property on sets of $k$-vector space transformations that is inherited by quotients. Suppose that for every nonzero $k$-vector space $\, V$ and every subset $X$ of $\, \End_k(V)$ that satisfies property $P$ there exists a $\, 1$-dimensional $X$-invariant subspace of $\, V$. Then given any $k$-vector space $\, V$, every subset of $\, \End_k(V)$ that satisfies property $P$ is triangularizable.
\end{lemma}

\begin{proof}
Let $V$ be a $k$-vector space and $X \subseteq \End_k(V)$ be such that $X$ satisfies property $P$. We shall show that $X$ is triangularizable. We may assume that $V \neq \{0\}$, since otherwise every subset of $\End_k(V)$ is triangularizable.

We begin by constructing recursively for each ordinal $\alpha$ an $X$-invariant subspace $V_\alpha \subseteq V$, and for each successor ordinal $\alpha$ a vector $v_\alpha \in V$. Set $V_0=\{0\}$. Now let $\alpha > 0$ be an ordinal and assume that $V_{\gamma}$ has been defined for every $\gamma < \alpha$. If $\alpha$ is a limit ordinal, then let $V_{\alpha} = \bigcup_{\gamma < \alpha} V_{\gamma}$. Since each $V_{\gamma}$ is assumed to be $X$-invariant, their union $V_{\alpha}$ will also be $X$-invariant. Next, if $\alpha$ is a successor ordinal, then let $\beta$ be its predecessor. By hypothesis, the set $\overline{X}$ of transformations on $V/V_{\beta}$ induced by $X$ satisfies property $P$. Thus, there is a $1$-dimensional $\overline{X}$-invariant subspace $W/V_{\beta}$ of $V/V_{\beta}$ (assuming that $V\neq V_{\beta}$). Let $v_{\alpha} \in V$ be such that $\{v_{\alpha}+V_{\beta}\}$ is a basis for $W/V_{\beta}$, and define $V_{\alpha} = \langle V_{\beta} \cup \{v_{\alpha}\}\rangle$. Then $V_{\alpha}$ must be $X$-invariant, because of the invariance of $V_{\beta}$ and $W/V_{\beta}$. We proceed in this fashion until $V = \bigcup_{\alpha \in \Lambda} V_{\alpha}$ for some ordinal $\Lambda$.

Now let $$\Gamma = \{\alpha \in \Lambda \mid \alpha \text{ is a successor ordinal}\},$$ and let $\B = \{v_{\alpha} \mid \alpha \in \Gamma\}$. As a subset of a well-ordered set, $\Gamma$ is itself well-ordered. Since we introduced new vectors only at successor steps in our construction, $$V = \bigcup_{\alpha \in \Gamma} V_{\alpha} =  \bigcup_{\alpha \in \Gamma} \langle \{v_{\gamma} \mid \gamma \leq \alpha, \gamma \in \Gamma\}\rangle,$$ and hence $V = \langle \B \rangle$. Since $V_{\alpha}/V_{\beta} = \langle v_{\alpha} +V_{\beta} \rangle$ is $1$-dimensional for all $\alpha \in \Gamma$ with predecessor $\beta$, we conclude that $\B$ is a basis for $V$. Also, since $V_{\alpha} = \langle \{v_{\gamma} \mid \gamma \leq \alpha, \gamma \in \Gamma\}\rangle$ is $X$-invariant for all $\alpha \in \Gamma$, it follows that $X(v_{\alpha}) \in \langle \{v_{\gamma} \mid \gamma \leq \alpha, \gamma \in \Gamma\} \rangle$ for all $\alpha \in \Gamma$. Thus,  $X$ is triangular with respect to $\B$, a basis for $V$ indexed by the well-ordered set $\Gamma$.
\end{proof}

Lemma~\ref{tri-lemma} gives the following alternative description of triangularizable sets, which additionally generalizes~\cite[Lemma 5]{ZM}.

\begin{proposition} \label{subspaces-prop}
Let $k$ be a field, $\, V$ a nonzero $k$-vector space, and $X \subseteq \End_k(V)$. Then $X$ is triangularizable if and only if there exists a well-ordered $\, ($by inclusion$)$ set of $X$-invariant subspaces of $\, V$, which is maximal as a well-ordered set of subspaces of $\, V$.
\end{proposition}

\begin{proof}
Suppose that $X$ is triangularizable. Then there is a well-ordered set $(\Omega, \leq)$ and a basis $\B = \{v_{\alpha} \mid \alpha \in \Omega\}$ for $V$ such that $T(v_{\alpha}) \in \langle \{v_{\beta} \mid \beta \leq \alpha\} \rangle$ for all $\alpha \in \Omega$ and all $T \in X$. Since every well-ordered set is order-isomorphic to an ordinal, we may assume that $\Omega$ is an ordinal. For each $\alpha \in \Omega$ set $V_{\alpha} = \langle \{v_{\beta} \mid \beta < \alpha\} \rangle$, where $V_0$ is understood to be the zero space ($0$ being the least element of $\Omega$). Then for all $\alpha_1, \alpha_2 \in \Omega$ we have $V_{\alpha_1} \subseteq V_{\alpha_2}$ if and only if $\alpha_1 \leq \alpha_2$.  Since $(\Omega, \leq)$ is well-ordered, it follows that $U=\{V_{\alpha} \mid \alpha \in \Omega^+\}$ is well-ordered by set inclusion, where $\Omega^+ = \Omega \cup \{\Omega\}$ is the successor of $\Omega$ and $V = V_{\Omega}$. Moreover, $X(v_{\beta}) \in V_{\alpha}$ for all $\alpha, \beta \in \Omega$ satisfying $\beta < \alpha$, from which it follows that each element of $U$ is $X$-invariant. It remains to show that $U$ is maximal. First, note that for each $\alpha \in \Omega^+$ we have $\bigcup_{\beta<\alpha} V_\beta \subseteq V_\alpha$, with equality if $\alpha$ is a limit ordinal, and $V_{\alpha}/(\bigcup_{\beta<\alpha} V_{\beta})$ $1$-dimensional otherwise.

Now, let $W \subseteq V$ be a subspace, which we may assume to be nonzero, that is comparable under set inclusion to $V_{\alpha}$ for each $\alpha \in \Omega^+$. Since $\Omega^+$ is well-ordered, there is a least $\alpha \in \Omega^+$ such that $W\subseteq V_{\alpha}$. Since $W$ is comparable to each element of $U$, from the choice of $V_{\alpha}$ it follows that $V_{\beta} \subset W$ for all $\beta < \alpha$. If $\alpha$ is a limit ordinal, then $V_{\alpha} = \bigcup_{\beta<\alpha} V_{\beta} \subseteq W$, and hence $W=V_{\alpha} \in U$. Otherwise, there is a $\beta \in \Omega^+$ such that $\alpha$ is the successor of $\beta$, and $V_{\beta} \subset W \subseteq V_{\alpha}$. But in this case $V_{\alpha} /V_{\beta}$ is $1$-dimensional, and therefore $W=V_{\alpha} \in U$ once again. Thus $U$ is a maximal well-ordered set of subspaces of $V$.

Conversely, suppose that there exists an ordinal $\Omega$ and a set $U = \{V_{\alpha} \mid \alpha \in \Omega\}$ of $X$-invariant subspaces of $V$, such that $V_{\alpha_1} \subseteq V_{\alpha_2}$ if and only if $\alpha_1 \leq \alpha_2$ (for all $\alpha_1, \alpha_2 \in \Omega$), and $U$ is maximal as a well-ordered set of subspaces of $V$. Since $V \neq \{0\}$, the set $U$ contains at least one nonzero subspace of $V$. Then letting $\alpha \in \Omega$ be the least element such that $V_{\alpha} \neq \{0\}$, the maximality of $U$ implies that $V_{\alpha}$ is an $X$-invariant subspace that is necessarily $1$-dimensional. 

By  Lemma~\ref{tri-lemma}, to conclude that $X$ is triangularizable, it suffices to show that the property of having a maximal well-ordered set of invariant subspaces is inherited by quotients. To that end, keeping $U$ as before, and letting $W \subseteq V$ be an $X$-invariant subspace, we shall show that $\overline{U} = \{(V_{\alpha} + W)/W \mid \alpha \in \Omega\}$ is a maximal well-ordered set of $\overline{X}$-invariant subspaces of $V/W$. 

For each $\alpha \in \Omega$, it follows immediately from $V_{\alpha}$ and $W$ being $X$-invariant that $(V_\alpha + W)/W$ is $\overline{X}$-invariant. Next, let $\alpha, \beta \in \Omega$. If $V_{\alpha} \subset V_{\beta}$, then $(V_{\alpha}+W)/W \subseteq (V_{\beta}+W)/W$. Conversely, if $(V_{\alpha}+W)/W \subset (V_{\beta}+W)/W$, then $V_{\alpha}+W \subset V_{\beta}+W$, and hence $V_{\alpha} \subset V_{\beta}$ (since $U$ being well-ordered implies that either $V_{\alpha} \subseteq V_{\beta}$ or $V_{\beta} \subseteq V_{\alpha}$). Thus, upon removing any repeated terms from $\overline{U}$, sending $(V_{\alpha}+W)/W$ to $V_{\alpha}$ gives an order-embedding of $\overline{U}$ into $U$, from which it follows that $\overline{U}$ is well-ordered. 

Finally, to show that $\overline{U}$ is maximal, let $Y$ be a subspace of $V$ containing $W$, such that $Y/W$ is comparable to each element of $\overline{U}$. Then we can find a least $\alpha \in \Omega$ such that $Y/W \subseteq (V_{\alpha}+W)/W$. If $\alpha$ is a limit ordinal, then $V_{\alpha} = \bigcup_{\beta<\alpha} V_{\beta}$, by the maximality of $U$, and hence $$(V_{\alpha}+W)/W = \bigcup_{\beta<\alpha} (V_{\beta}+W)/W \subseteq Y/W$$ implies that $Y/W = (V_{\alpha}+W)/W$. Thus, let us assume that $\alpha$ is a successor ordinal, with predecessor $\beta$. Then $V_{\beta}+W \subset Y \subseteq V_{\alpha}+W$. Writing $Y' = Y \cap V_{\alpha}$, we have $V_{\beta} \subset Y' \subseteq V_{\alpha}$, since $V_{\beta} \subset V_{\alpha}$. Hence $Y' = V_{\alpha}$, by the maximality of $U$, and therefore $Y/W = (V_{\alpha}+W)/W$ once again. Thus $\overline{U}$ is maximal as a well-ordered set of subspaces of $V/W$.
\end{proof}

We conclude this section with a list of properties that are inherited by quotients, which will be useful later. These are all standard or easy to prove.

\begin{lemma} \label{pass-to-quotients}
Let $k$ be a field, $\, V$ a $k$-vector space, $X \subseteq \End_k(V)$, and $\, W$ an $X$-invariant subspace of $\, V$. For each of the following properties, if $X$ satisfies it, then so does $\overline{X} \subseteq \End_k(V/W)$. 
\begin{enumerate}
\item[$(1)$] The set is closed under addition.
\item[$(2)$] The set is closed under multiplication $($i.e., composition of transformations$)$.
\item[$(3)$] The set is closed under scalar multiplication.
\item[$(4)$] The set is topologically nilpotent.
\end{enumerate}
\end{lemma}

\begin{proof}
We may assume that $X \neq \emptyset$, since otherwise the conditions above are vacuously true for both $X$ and $\overline{X}$. Let $T,S \in X$. Then for all $v \in V$ we have $$(\overline{T}+\overline{S})(v +W) = \overline{T}(v +W) + \overline{S}(v +W) = (T(v)+W) + (S(v)+W)$$ $$= (T(v) + S(v))+W = (T+S)(v) + W = (\overline{T+S})(v+W),$$ showing that $\overline{T}+\overline{S} = \overline{T+S}$. Analogous computations also show that $\overline{T} \cdot \overline{S} = \overline{T \cdot S}$ and $\overline{aT} = a\overline{T}$ for all $a\in k$. From this it follows immediately that if $X$ is closed under addition, multiplication, or scalar multiplication, then so is $\overline{X}$.

Finally, suppose that $X$ is topologically nilpotent, let $T_1, T_2, T_3, \cdots \in X$, and let $U$ be a subspace of $V$ containing $W$, such that $U/W$ is finite-dimensional. Also, let $Y$ be a finite-dimensional subspace of $V$ such that $U+W = Y+W$. Then $T_n\cdots T_2T_1(Y) = 0$ for some $n \in \Z^+$. Hence $$\overline{T_n}\cdots \overline{T_2}\cdot \overline{T_1}(U/W) = \overline{T_n\cdots T_2T_1}(U/W) = T_n\cdots T_2T_1(Y) +W = W,$$ from which it follows that $\overline{X}$ is topologically nilpotent.
\end{proof}

\section{Topologically Nilpotent Sets}

Using the triangularization lemma of the previous section, we can give an alternative characterization an arbitrary strictly triangularizable subset of $\End_k(V)$, which generalizes the equivalence of (1) and (3) in Proposition~\ref{top-nilpt-prop}.

\begin{theorem} \label{strictly-upper}
Let $k$ be a field, $\, V$ a nonzero $k$-vector space, and $X \subseteq \End_k(V)$. Then $X$ is strictly triangularizable if and only if $X$ is topologically nilpotent.
\end{theorem}

\begin{proof}
We may assume that $X \neq \emptyset$, since otherwise $X$ is both strictly triangularizable and topologically nilpotent, by the definitions of those terms.

Suppose that $X$ is strictly triangular with respect to a well-ordered basis $(\B, \leq)$ for $V$, let $T_1, T_2, T_3, \ldots \in X$, and let $v \in \B$. Seeking a contradiction, suppose also that $T_n \cdots T_2T_1(v) \neq 0$ for all $n \in \Z^+$. Since $T_1$ is strictly triangular with respect to $\B$, we have $T_1(v) = \sum_{w \leq u_1} a_ww$ for some $a_w \in k$ and $u_1, w \in \B$, where $u_1 < v$ and $a_{u_1} \neq 0$. Similarly, $$T_2T_1(v) = \sum_{w \leq u_1} a_wT_2(w) = \sum_{w' \leq u_2} b_{w'}w'$$ for some $b_{w'} \in k$ and $u_2, w' \in \B$, where $u_2 < u_1$ and $b_{u_2} \neq 0$. Continuing in this fashion, for each $n \in \Z^+$ we can write $T_n \cdots T_2T_1(v) = \sum_{w \leq u_n} c_ww$  for some $c_w \in k$ and $u_n, w \in \B$, where $$u_n < \cdots < u_2 < u_1$$ and $c_{u_n} \neq 0$. Thus we obtain an infinite descending chain $$\cdots < u_3 < u_2 < u_1$$ of elements of $\B$, which contradicts the hypothesis that $\B$ is well-ordered. Therefore it must be the case that $T_n \cdots T_2T_1(v) = 0$ for some $n \in \Z^+$. Since $v \in \B$ was arbitrary and the $T_i$ are linear, it follows that for every finite-dimensional subspace $W$ of $V$ there exists $n \in \Z^+$ such that $T_n \cdots T_2T_1(W) = 0$. Since the transformations $T_i \in X$ were arbitrary, this shows that $X$ is topologically nilpotent.

Conversely, suppose that $X$ is topologically nilpotent. It suffices to show that $X$ is triangularizable, since by hypothesis every element of $X$ is topologically nilpotent, and hence strictly triangular with respect to any well-ordered basis for $V$ with respect to which it is triangular, by Proposition~\ref{top-nilpt-prop}. Since, according to Lemma~\ref{pass-to-quotients}(4), the property of being topologically nilpotent is inherited by quotients, by Lemma~\ref{tri-lemma} it in turn suffices to show that there is a $1$-dimensional $X$-invariant subspace of $V$. To find such a $1$-dimensional subspace, we shall show that  $\bigcap_{T \in X} \ker(T) \neq \{0\}$. For, every ($1$-dimensional) subspace of $\bigcap_{T \in X} \ker(T)$ is necessarily $X$-invariant.

Seeking a contradiction, suppose that $\bigcap_{T \in X} \ker(T) = \{0\}$. Then necessarily $X \neq \{0\}$, since we have assumed that $V$ is nonzero. Hence we can find $T_1 \in X$ and $v \in V$ such that $T_1(v) \neq 0$. Again, since $\bigcap_{T \in X} \ker(T) = \{0\}$, we can find $T_2 \in X$ such that $T_2T_1(v) \neq 0$. Continuing in this fashion, we obtain $T_1, T_2, T_3, \ldots \in X$ such that $T_n\cdots T_2T_1(v) \neq 0$ for all $n \in \Z^+$, which contradicts $X$ being topologically nilpotent. Hence $\bigcap_{T \in X} \ker(T) \neq \{0\}$, as desired.
\end{proof}

Next we note that Theorem~\ref{strictly-upper} generalizes~\cite[Theorem 2]{Chebotar}. The algebra $R$ in this statement is necessarily nonunital, unless it is zero.

\begin{corollary}[Chebotar] \label{Cheb-thrm}
Let $k$ be a field, $\, V$ a nonzero $k$-vector space, and $R$ a finite-dimensional nilpotent $k$-subalgebra of $\, \End_k(V)$. Then $R$ is strictly triangularizable.
\end{corollary}

\begin{proof}
Since $R$ is nilpotent, it is trivially topologically nilpotent, and hence strictly triangularizable, by Theorem~\ref{strictly-upper}.
\end{proof}

Another immediate consequence of Theorem~\ref{strictly-upper} is a well-known result from~\cite{Levitzki}.

\begin{corollary}[Levitzki]
Let $k$ be a field, $n \in \Z^+$, and $X$ a nilpotent multiplicative subsemigroup of $\, \M_n(k)$. Then $X$ is triangularizable.
\end{corollary}

We conclude this section with one more application of the theorem.

\begin{corollary} \label{nilp-alg-cor}
Let $k$ be a field, $\, V$ a nonzero $k$-vector space, $X \subseteq \End_k(V)$, and $R$ the nonunital $k$-subalgebra of $\, \End_k(V)$ generated by $X$. If $X$ is topologically nilpotent, then so is $R$.
\end{corollary}

\begin{proof}
We may assume that $X \neq \emptyset$, since otherwise $R=\{0\}$. Supposing that $X$ is topologically nilpotent, by Theorem~\ref{strictly-upper}, there is a well-ordered basis $(\B, \leq)$ for $V$ with respect to which $X$ is strictly triangular. Now let $T \in R$ be any element. Then $T$ can be expressed as $T=\sum_{i=1}^n a_iS_{i,1}\cdots S_{i,m_i}$ for some $a_i \in k$, $S_{i,j} \in X$, and $n, m_i \in \Z^+$. Since each $S_{i,j}$ is strictly triangular with respect to $\B$, the same is true of the products $S_{i,1}\cdots S_{i,m_i}$, and hence also of $T$. Thus $R$ is strictly triangular with respect to $\B$, and therefore topologically nilpotent, by Theorem~\ref{strictly-upper}.
\end{proof}

\section{Radicals} \label{rad-sect}

The remainder of the paper is concerned with triangularizable subrings of $\End_k(V)$. The main goal of this section is to show that the Jacobson radical $\rad(R)$ of such a subring $R$ of $\End_k(V)$ is topologically nilpotent. We make the following definition to facilitate the discussion.

\begin{definition}
Let $k$ be a field, $\, V$ a nonzero $k$-vector space, and $R$ a subring of $\, \End_k(V)$. Set
$$\tnil(R) = \{T \in R \mid T \text{ is topologically nilpotent}\}.$$
\end{definition}

The next lemma is an analogue of the standard fact that every nil left (or right) ideal of a ring is contained in its Jacobson radical (see, e.g.,~\cite[Lemma 4.11]{Lam}).

\begin{lemma} \label{nil-rad-lem}
Let $k$ be a field, $V$ a nonzero $k$-vector space, $R$ a subring of $\, \End_k(V)$ that is closed in the function topology, and $J$ a left $\, ($or right$)$ ideal of $R$. If $J \subseteq \tnil(R)$, then $J \subseteq \rad (R)$.
\end{lemma}

\begin{proof}
We shall only treat the case where $J$ is a left ideal, since the right ideal version is entirely analogous.

Let $T \in J$, and let $S \in R$. We shall show that $1 - ST$ is invertible in $R$. Since $S\in R$ is arbitrary, this implies that $T \in \rad(R)$ (see, e.g.,~\cite[Lemma 4.1]{Lam}), and hence $J \subseteq \rad (R)$.

To show that $1 - ST$ is invertible, we first note that $ST \in J$, and hence $ST$ is topologically nilpotent, assuming that $J \subseteq \tnil(R)$. Thus, for any $v \in V$ there exists $n\in \Z^+$ such that $(ST)^i(v) = 0$ for all $i \geq n$. It follows that $\sum_{i=0}^\infty (ST)^i$ converges to transformation in the function topology on $\End_k(V)$. Since $\sum_{i=0}^m (ST)^i \in R$ for all $m \in \Z^+$, since every open neighborhood of $\sum_{i=0}^\infty (ST)^i$ contains such a finite sum, and since $R$ is assumed to be closed, we have $\sum_{i=0}^\infty (ST)^i \in R$. 

Now, for any $v \in V$, we can find $n\in \Z^+$ such that $(ST)^n(v) = 0 = (ST)^n(1 - ST)(v)$. Then $$(1 - ST)\sum_{i=0}^\infty (ST)^i(v) = (1 - ST)\sum_{i=0}^{n-1} (ST)^i(v) = (1 - (ST)^{n})(v) = v$$ and $$\sum_{i=0}^\infty (ST)^i(1 - ST)(v) = \sum_{i=0}^{n-1} (ST)^i(1 - ST)(v) = (1 - (ST)^{n})(v) = v.$$ Hence $(1-ST)^{-1} = \sum_{i=0}^\infty (ST)^i \in R$.
\end{proof}

The next example shows the necessity of assuming that $R$ is closed in this lemma.

\begin{example}
Let $k$ be a field, and let $V$ be a $k$-vector space with basis $\B = \{v_i \mid i \in \Z^+\}$. Define $T \in \End_k(V)$ by 
$$T(v_i) = 
\left\{ \begin{array}{ll}
v_{i-1} & \text{if } \, i > 1\\
0 & \text{if } \, i = 1
\end{array}\right.,$$
and extend linearly to all of $V$. Also, let $R$ be the $k$-subalgebra of $\End_k(V)$ generated by $T$. It is easy to see that $p(T) \neq 0$ for any nonzero polynomial $p(x) \in k[x]$. (If $p(x) = \sum_{i=0}^n a_i x^i$ for some $n \in \Z^+$ and $a_i \in k$, then $0 = p(T)(v_{n+1}) = \sum_{i=0}^n a_i v_{n+1-i}$ would imply that $a_0=\cdots = a_n=0$, since $\{v_1, v_2, \dots, v_{n+1}\}$ is linearly independent.) Hence $R \cong k[x]$ as $k$-algebras, via a map that sends $T$ to $x$, and in particular $\rad(R)=0$. Now let $J$ be the ideal of $R$ generated by $T$. Then for all $\sum_{i=1}^n a_i T^i \in J$ and for all $m \in \Z^+$ we have $(\sum_{i=1}^n a_i T^i)^{m}(v_m) = 0$. Hence $J \subseteq \tnil(R)$, but $J \not\subseteq \rad(R)$. (In fact, $J = \tnil(R)$.)

We conclude by noting that $R$ is not closed in the function topology. This follows from the fact that $\sum_{i=0}^\infty T^i$ converges to a transformation in $\End_k(V)$ (by the same reasoning as in the proof of Lemma~\ref{nil-rad-lem}), and every open neighborhood of $\sum_{i=0}^\infty T^i$ contains an element of $R$ (of the form $\sum_{i=0}^n T^i$), but $\sum_{i=0}^\infty T^i \notin R$.
\hfill $\Box$
\end{example}

If a subring $R$ of $\End_k(V)$ is triangularizable, then we can say much more about $\tnil(R)$ and its relationship with $\rad (R)$ than we did in Lemma~\ref{nil-rad-lem}.

\begin{proposition} \label{nil-rad-prop}
Let $k$ be a field, $V$ a nonzero $k$-vector space, and $R$ a subring of $\, \End_k(V)$ triangular with respect to a well-ordered basis $\, (\B, \leq)$ for $\, V$. Then the following hold.
\begin{enumerate}
\item[$(1)$] $\, \tnil(R)$ is the ideal consisting precisely of the transformations in $R$ that are strictly triangular with respect to $\, (\B, \leq)$.
\item[$(2)$] $\, \tnil(R)$ is topologically nilpotent.
\item[$(3)$] $\, \rad (R) \subseteq \tnil(R)$.
\end{enumerate}  
Moreover, if $R$ is closed in the function topology, then $\, \tnil(R) = \rad (R)$.
\end{proposition}

\begin{proof}
(1) It follows from Proposition~\ref{top-nilpt-prop} that $\tnil(R)$ consists precisely of the transformations in $R$ that are strictly triangular with respect to $(\B, \leq)$, and in particular, $\tnil(R) \neq \emptyset$. To show that this set is an ideal of $R$, let $T \in \tnil(R)$ and $S_1, S_2 \in R$. Since $T$ is strictly triangular with respect to $\B$, for any $v \in \B$ we have $$S_1TS_2(v) \in S_1T(\langle \{u \in \B \mid u \leq v\} \rangle) \subseteq S_1(\langle \{u \in \B \mid u < v\} \rangle) \subseteq \langle \{u \in \B \mid u < v\} \rangle,$$ and therefore $S_1TS_2$ is also strictly triangular with respect to $\B$. Since the sum of any two transformations in $R$ that are strictly triangular with respect to $\B$ is also strictly triangular with respect to $\B$, we see that $\tnil(R)$ is an ideal of $R$.

(2) Since $\tnil(R)$ is strictly triangularizable, by (1), Theorem~\ref{strictly-upper} implies that it is topologically nilpotent.

(3) By (1), it suffices to show that every element of $\rad (R)$ is strictly triangular with respect to $\B$. Seeking a contradiction, suppose that $T \in \rad (R)$ is not strictly triangular with respect to $\B$. Thus $\pi_vT(v) = av$ for some $v \in \B$ and $a \in k \setminus \{0\}$, where $\pi_v \in \End_k(V)$ is the projection onto $\langle v \rangle$ with kernel $\langle \B \setminus \{v\} \rangle$. Hence $$\pi_v(1-a^{-1}T)(v) = \pi_v(v) - a^{-1}\pi_vT(v)=v-v=0,$$ and so $\pi_v(1-a^{-1}T)\pi_v = 0$. By Lemma~\ref{inverse-prop}, this shows that $1-a^{-1}T$ is not invertible (in $\End_k(V)$, and hence also in $R$), contradicting $T \in \rad (R)$. Thus every element of $\rad (R)$ is strictly triangular with respect to $\B$, and therefore $\rad (R) \subseteq \tnil(R)$.

Finally, if $R$ is closed, then $\tnil(R) \subseteq \rad (R)$, by Lemma~\ref{nil-rad-lem}, since $\tnil(R)$ is an ideal of $R$, by (1). Hence, in this case $\tnil(R) = \rad(R)$, by (3).
\end{proof}

The next example shows that $\tnil(R)$ need not be topologically nilpotent in general.

\begin{example} \label{lower-tri-eg}
Let $k$ be a field, let $V$ be a countably infinite-dimensional $k$-vector space, and identify $\End_k(V)$ with the set of column-finite infinite matrices (indexed by the positive integers). Also, let $X \subseteq \End_k(V)$ be the set of all strictly lower-triangular matrices with only finitely many nonzero entries. Thus an arbitrary element of $X$ has the following form.
$$\left(\begin{array}{cccccc}
0 & 0 & 0 & 0 & 0 & \cdots \\
* & 0 & 0 & 0 & 0 & \cdots \\
* & * & 0 & 0 & 0 & \cdots \\
\vdots & \vdots & \vdots & \ddots & \ddots & \ddots \\
* & * & \cdots & * & 0 & \cdots \\
0 & 0 & \cdots & 0 & 0 & \cdots \\
\vdots & \vdots & \vdots & \vdots & \vdots & \ddots \\
\end{array}\right)$$
It is easy to see that $X$ is a nonunital $k$-subalgebra of $\End_k(V)$, where every element is nilpotent. Thus, in particular, each element of $X$ is strictly triangularizable, by Proposition~\ref{top-nilpt-prop}. Letting $R$ be the unital $k$-subalgebra of $\End_k(V)$ generated by $X$, we have $X = \tnil(R)$, and this set is clearly an ideal of $R$.

Next we wish to show that $X = \tnil(R)$ is not topologically nilpotent. For all $i,j \in \Z^+$ let $E_{i,j}$ be the matrix unit with a $1$ in the $i$-th row and $j$-th column, and zeros elsewhere. Then $E_{i, i-1} \in \tnil(R)$ for all $i \geq 2$, and it is easy to see that $E_{n, n-1}\cdots E_{3,2}E_{2,1}= E_{n,1}$ for all $n \geq 2$. Hence, letting
$$v = \left(\begin{array}{c}
1 \\
0 \\
0 \\
\vdots \\
\end{array}\right),$$ 
we have $E_{n, n-1}\cdots E_{3,2}E_{2,1}(v) = E_{n,1}(v) \neq 0$ for all $n \geq 2$, and therefore $\tnil(R)$ is not topologically nilpotent. 

It follows from Proposition~\ref{nil-rad-prop} that $R$ is not triangularizable (as we have defined the term). This can also be shown directly, by noting that there is no $1$-dimensional $R$-invariant subspace of $V$. (If $R$ were triangular with respect to a well-ordered basis $(\B, \leq)$ for $V$, then the subspace spanned by the least element of $\B$ would necessarily be $R$-invariant.) For suppose, seeking a contradiction, that there exists $v \in V \setminus \{0\}$ such that $R(\langle v \rangle) \subseteq \langle v \rangle$. We can write 
$$v = \left(\begin{array}{c}
a_1 \\
a_2 \\
\vdots \\
a_n \\
0 \\
\vdots \\
\end{array}\right),$$
where $a_1, \dots, a_n \in k$ and $a_n \neq 0$. Then $E_{n+1,n}(v)$ has $a_n$ in the $(n+1)$-st coordinate and zeros elsewhere. Thus $E_{n+1,n} \in R$, but $E_{n+1,n}(v) \notin \langle v \rangle$, giving the desired contradiction. 

We conclude by noting that $X = \rad(R)$. Since $R/X \cong k$ as $k$-algebras, $X$ is a maximal left ideal, and hence $\rad(R) \subseteq X$. Also, for any $T \in X$ and any $S \in R$, 
$$1-ST = \left(\begin{array}{cc}
M & 0 \\
0 & \mathbf{1} \\
\end{array}\right),$$ 
where $M$ is a lower-triangular (finite) matrix with $1$s on the main diagonal, and $\mathbf{1}$ is an infinite matrix having $1$s on the main diagonal and $0$s elsewhere. Since $M$ is invertible, by finite-dimensional linear algebra, so is $1-ST$, and hence $T \in \rad(R)$. It follows that $\rad(R) = X = \tnil(R)$. \hfill $\Box$
\end{example}

\section{Triangularizable Algebras} \label{tri-alg-sect}

This section is devoted to describing the triangularizable subalgebras of $\End_k(V)$. We begin with a technical lemma.

\begin{lemma} \label{idempt-lemma}
Let $k$ be a field, $V$ a nonzero $k$-vector space, $R$ a $k$-subalgebra of $\, \End_k(V)$, and $I$ a topologically nilpotent ideal of $R$. Suppose that there exists $\, \{E_{\alpha} \mid \alpha \in \Omega\} \subseteq R$ such that the following conditions are satisfied. 
\begin{enumerate}
\item[$($a$)$] $E_{\alpha}E_{\alpha} - E_{\alpha} \in I$ for all $\alpha \in \Omega$, and $E_{\alpha}E_{\beta} \in I$ whenever $\alpha \neq \beta$. 
\item[$($b$)$] For every $T \in R$ and $v \in V$, there exist $\alpha_1, \dots, \alpha_n \in \Omega$, $a_{\alpha_1}, \dots, a_{\alpha_n} \in k$, and $S \in I$ such that $T(v) =\sum_{i=1}^n a_{\alpha_i} E_{\alpha_i}(v) +S(v)$. 
\end{enumerate}
Then $R$ is triangularizable.
\end{lemma}

\begin{proof}
Since $I$ is topologically nilpotent, it is (strictly) triangularizable, by Theorem~\ref{strictly-upper}. Hence there exists a $1$-dimensional $I$-invariant subspace $W \subseteq V$ (e.g., by Proposition~\ref{subspaces-prop}), where necessarily $I(W) = \{0\}$. It follows that $U = \bigcap_{T \in I} \ker(T) \neq \{0\}$. Moreover, for all $T\in R$ we have $T(U) \subseteq U$, since $ST \in I$ for all $S \in I$, and hence $ST(U) = \{0\}$. 

We also note that $E_{\alpha}(U) \neq \{0\}$ for some $\alpha \in \Omega$. For if it were the case that $E_{\alpha}(U) = \{0\}$ for all $\alpha \in \Omega$, then the fact that $I(U) = \{0\}$ and condition (b) would imply that $R(U) = \{0\}$. Since $R$ is assumed to be unital, this would mean that $U=\{0\}$, contrary to construction. Thus, we may choose $u \in U \setminus \{0\}$ and $\beta \in \Omega$ such that $v := E_{\beta}(u) \neq 0$. We claim that $\langle v \rangle$ is $R$-invariant. 

Since $I(U) = \{0\}$, the first relation in condition (a) gives $0=E_{\beta}E_{\beta}(u) - E_{\beta}(u)$, and hence $E_{\beta}(v)=v$. Likewise, the second relation in condition (a) gives $0=E_{\alpha}E_{\beta}(u) = E_{\alpha}(v)$ for all $\alpha \in \Omega \setminus \{\beta\}$. Moreover, $I(v) = 0$, since $v = E_{\beta}(u) \in U$, as noted in the first paragraph. Now let $T \in R$, and let $\alpha_1, \dots, \alpha_n \in \Omega$, $a_{\alpha_1}, \dots, a_{\alpha_n} \in k$, and $S \in I$ be such that $T(v) = \sum_{i=1}^n a_{\alpha_i} E_{\alpha_i}(v) +S(v)$, using condition (b). Then $T(v) = a_{\beta}E_{\beta}(v) =a_{\beta}v$ for some $a_{\beta} \in k$ (where $a_{\beta}=0$ if $\beta \notin \{\alpha_1, \dots, \alpha_n\}$), showing that the $1$-dimensional space $\langle v \rangle$ is $R$-invariant.

Finally, it follows from Lemma~\ref{pass-to-quotients} that the property of being a $k$-algebra with a topologically nilpotent ideal is inherited by quotients (see Definition~\ref{quotient-def}). Moreover, it is easy to see that the property of having a subset satisfying conditions (a) and (b), with respect to the relevant topologically nilpotent ideal, is also inherited by quotients. Hence, $R$ is triangularizable, by Lemma~\ref{tri-lemma}.
\end{proof}

In the results that follow we view subrings of $\End_k(V)$ as topological rings via the subspace topology induced by the function topology on $\End_k(V)$, and we view the product $k^{\Omega}$ as a topological ring via the product topology, where each component copy of $k$ is taken to be discrete. Also, see Section~\ref{top-sect} for a review of the quotient topology and isomorphisms of topological algebras.

\begin{proposition} \label{prod-tri-prop}
Let $k$ be a field, $V$ a nonzero $k$-vector space, and $R$ a $k$-subalgebra of $\, \End_k(V)$. Suppose there exist a topologically nilpotent ideal $I$ of $R$ and a set $\, \Omega$ such that $R/I \cong k^{\Omega}$ as topological $k$-algebras. Then $R$ is triangularizable, and $I = \tnil(R)$.
\end{proposition}

\begin{proof}
Let $\pi : R \to R/I$ be the canonical projection. Then, by the definition of the quotient topology, $\pi$ is continuous. Moreover, it is a standard fact that $\pi$ is necessarily an open map. (For if $\mathcal{U} \subseteq R$ is an open set, then $\mathcal{U}+I = \bigcup_{i \in I} (\mathcal{U}+i)$, and each $\mathcal{U}+i$ is open in $R$, by the continuity of addition. Hence $\pi^{-1}(\pi (\mathcal{U}))=\mathcal{U}+I$ is open, and therefore so is $\pi (\mathcal{U})$, again by the definition of the quotient topology.)

Since $R/I \cong k^{\Omega}$ as topological $k$-algebras, there must exist $\{E_{\alpha} \mid \alpha \in \Omega\} \subseteq R$, such that $\{E_{\alpha} +I \mid \alpha \in \Omega\} \subseteq R/I$ is a set of orthogonal idempotents, and $R/I$ is the closure of the $k$-vector space spanned by this set. Now let $T \in R$ and let $\mathcal{U} \subseteq R$ be an open neighborhood of $T$. Then, by the previous paragraph, $\pi(\mathcal{U})$ is an open neighborhood of $\pi(T)$, and hence there exist $\alpha_1, \dots, \alpha_n \in \Omega$ and $a_{\alpha_1}, \dots, a_{\alpha_n} \in k$ such that $\sum_{i=1}^n a_{\alpha_i} E_{\alpha_i} +I \in \pi (\mathcal{U})$. It follows that $\sum_{i=1}^n a_{\alpha_i} E_{\alpha_i} +S \in \mathcal{U}$ for some $S \in I$. Hence every open neighborhood of $T \in R$ contains an element of the form  $\sum_{i=1}^n a_{\alpha_i} E_{\alpha_i} +S$. In particular, for all $v\in V$ there exist $\alpha_1, \dots, \alpha_n \in \Omega$, $a_{\alpha_1}, \dots, a_{\alpha_n} \in k$, and $S \in I$ such that $T(v) = \sum_{i=1}^n a_{\alpha_i} E_{\alpha_i}(v) +S(v)$. Thus $\{E_{\alpha} \mid \alpha \in \Omega\}$ satisfies the conditions (a) and (b) of Lemma~\ref{idempt-lemma}, and therefore $R$ is triangularizable.

Next, we note that $I \subseteq \tnil(R)$, since every element of $I$ is topologically nilpotent. To prove the opposite inclusion, let $T \in \tnil(R)$. Then for every open neighborhood $\mathcal{U} \subseteq R$ of $0$ there is some $n \in \Z^+$ such $T^n \in \mathcal{U}$. We claim that $T + I \in R/I$ enjoys the same property. Letting $\mathcal{U} + I \subseteq R/I$ be an open neighborhood of $I$, where $\mathcal{U} \subseteq R$, the continuity of the canonical projection $\pi : R \to R/I$ implies that $\mathcal{U} + I$ is an open neighborhood of $0$ in $R$. Hence there is some $n \in \Z^+$ such $T^n \in \mathcal{U} + I$, and therefore $$(T+I)^n = T^n + I \in \mathcal{U} + I \subseteq R/I.$$ But in $R/I \cong k^{\Omega}$ the only element $r$ having this property (that for every open neighborhood $\mathcal{U}$ of $0$ there is some $n \in \Z^+$ such $r^n \in \mathcal{U}$) is $0$. Thus $T + I = I$, and hence $T \in I$. It follows that $I = \tnil(R)$.
\end{proof}

We observe that Example~\ref{lower-tri-eg} shows that the conclusion of the previous proposition would cease to hold if we were to remove the hypothesis that $I$ is topologically nilpotent. In that example we constructed a subalgebra $R$ of $\End_k(V)$ such that $R/\rad(R) \cong k$ as $k$-algebras, but $R$ is not triangularizable and $\rad(R)$ is not topologically nilpotent. (Note that $R/\rad(R)$ and $k$ in Example~\ref{lower-tri-eg} are both discrete, and so the isomorphism between them is necessarily topological.)

We now turn to giving necessary conditions for subalgebras of $\End_k(V)$ to be triangularizable.

\begin{proposition} \label{hom-prop}
Let $k$ be a field, $\, V$ a nonzero $k$-vector space, and $R$ a $k$-subalgebra of $\, \End_k(V)$ triangular with respect to a well-ordered basis $\, (\B, \leq)$. Define $\phi: R \to k^\B$ by $\phi(T) = (a_v)_{v \in \B}$, where $a_v \in k$ is such that $\pi_vT(v) = a_vv$, and $\pi_v \in \End_k(V)$ denotes the projection onto $\, \langle v \rangle$ with kernel $\, \langle \B \setminus \{v\} \rangle$. Then the following hold.
\begin{enumerate}
\item[$(1)$] The map $\phi: R \to k^\B$ is a continuous $k$-algebra homomorphism.
\item[$(2)$] $\, \tnil(R) = \ker (\phi)$, and this ideal is closed in the induced topology on $R$.
\item[$(3)$] If $R$ is the $k$-subalgebra of $\, \End_k(V)$ consisting of all the transformations triangular with respect to $\, (\B, \leq)$, then $\phi$ is open and surjective.
\end{enumerate}
In particular, $R/\tnil(R)$ is commutative.
\end{proposition}

\begin{proof}
(1) Let $S,T \in R$ and $v\in \B$. Since $S$ is triangular with respect to $\B$, we can write $S(v) = b_vv + \sum_{u<v} b_uu$ for some $b_v,b_u \in k$ and $u \in \B$. Since $T$ is triangular with respect to $\B$, we have $$\pi_vTS(v) = \pi_vT(b_vv) + \pi_vT\bigg(\sum_{u<v} b_uu\bigg) = b_v\pi_vT(v) = a_vb_vv,$$ where $a_v \in k$ is defined by $\pi_vT(v) = a_vv$. It follows that $\phi(TS) = \phi(T)\phi(S)$. It is also straightforward to show that $\phi$ is $k$-linear, and therefore $\phi$ is a $k$-algebra homomorphism.

Now let $\mathcal{U} \subseteq k^{\B}$ be a basic open set. Thus there exist $v_1, \dots, v_n \in \B$ and $b_{v_1}, \dots, b_{v_n} \in k$ such that $$\mathcal{U} = \{(c_v)_{v\in \B} \mid c_{v_i}=b_{v_i} \text{ for } 1\leq i\leq n\}.$$ Then 
\begin{align*}
\phi^{-1}(\mathcal{U}) & = \{S \in R \mid \pi_{v_i}S(v_i) = b_{v_i}v_i  \text{ for } 1\leq i\leq n\}\\
& = \bigcap_{i=1}^n \bigcup_{w \in \langle \{u \in \B \mid u < v_i\}\rangle} \{S \in R \mid S(v_i) = b_{v_i}v_i +w\}
\end{align*}
is open in the induced topology on $R$, from which it follows that $\phi$ is continuous.

(2) For any $T \in R$ we have
\begin{eqnarray*}
\phi(T) = 0 & \Leftrightarrow & \pi_vT(v)=0  \text{ for all } v\in \B \\
& \Leftrightarrow & T(v) \in \langle \{u \in \B \mid u < v\}\rangle \text{ for all } v\in \B \\
& \Leftrightarrow & T \in \tnil(R) \ (\text{by Proposition}~\ref{top-nilpt-prop}),
\end{eqnarray*}
and hence $\tnil (R) = \ker (\phi)$. Since $\{0\}$ is closed in $k^\B$ (which, as a product of discrete spaces, is Hausdorff), $\ker (\phi) = \phi^{-1}(\{0\})$, and $\phi$ is continuous, we conclude that $\ker (\phi)$ is closed in the induced topology on $R$.

(3) Suppose that $R$ consists of all the transformations in $\End_k(V)$ triangular with respect to $\B$, and let $\mathcal{U} \subseteq R$ be an open subset. We shall show that $\phi(\mathcal{U})$ is open. 

We may assume that $$\mathcal{U} = \{S \in R \mid S(v_i) = w_i \text{ for } 1\leq i\leq n\}$$ for some $v_1, \dots, v_n \in \B$ and  $w_1, \dots, w_n \in V$, since it is easy to see that every open set in $R$ is a union of sets of this form. Write $\pi_{v_i}(w_i) = b_{v_i} v_i$ $(1\leq i\leq n)$ for some $b_{v_i} \in k$. Given a $c_v \in k$ for each $v \in \B$, such that $c_{v_i}=b_{v_i}$ for $1\leq i\leq n$, define $T \in \End_k(V)$ by $T(v_i) = w_i$ for $1\leq i\leq n$, and $T(v)=c_vv$ for all $v \in \B\setminus \{v_1, \dots, v_n\}$. Since $T$ agrees with certain elements of $R$ on $\{v_1, \dots, v_n \}$, and each element of $\B\setminus \{v_1, \dots, v_n\}$ is an eigenvector of $T$, it is triangular with respect to $\B$, and hence $T \in R$. Thus $T \in \U$ and $\phi(T) = (c_v)_{v\in \B}$. It follows that $$\phi(\mathcal{U}) = \{(c_v)_{v\in \B}  \in k^{\B} \mid c_{v_i}=b_{v_i} \text{ for } 1\leq i\leq n\},$$ which is open in $k^{\B}$. Hence $\phi$ is open, and it is surjective, by a simpler version of the same argument.

Finally, $R/\tnil(R)$ is isomorphic to a subalgebra of $k^{\B}$, by (1) and (2), and is hence commutative.
\end{proof}

Combining the previous two propositions gives our main result.

\begin{theorem} \label{main-theorem}
Let $k$ be a field, $\, V$ a nonzero $k$-vector space, and $R$ a $k$-subalgebra of $\, \End_k(V)$. Then the following are equivalent.
\begin{enumerate}
\item[$(1)$] $R$ is triangularizable.
\item[$(2)$] $R$ is contained in a $k$-subalgebra $A$ of $\, \End_k(V)$ such that $A/\rad(A) \cong k^{\Omega}$ as topological $k$-algebras for some set $\, \Omega$, and $\, \rad (A)$ is topologically nilpotent.
\end{enumerate}
Moreover, if $\,(2)$ holds and $R$ is closed in the function topology, then $\, \rad(R) = R \cap \rad (A)$.
\end{theorem}

\begin{proof}
Suppose that $R$ is triangular with respect to some well-ordered basis $(\B, \leq)$ for $V$. Let $A$ be the $k$-subalgebra of $\End_k(V)$ consisting of all the transformations triangular with respect to $(\B, \leq)$, and let $\phi : A \to k^\B$ be as in Proposition~\ref{hom-prop}. Then, by that proposition, $\phi$ is a surjective open continuous $k$-algebra homomorphism, such that $\ker(\phi) = \tnil(A)$. Hence $A/\tnil(A) \cong k^{\B}$ as topological $k$-algebras (see, e.g.,~\cite[Theorem 5.11]{Warner}). Now, by Lemma~\ref{closure-lemma}, the closure $\cl(A)$ of $A$ in the function topology on $\End_k(V)$ is triangular with respect to $\B$. Hence $\cl(A) \subseteq A$, and so $A$ is closed in $\End_k(V)$. Therefore, by Proposition~\ref{nil-rad-prop}, $\tnil(A)=\rad(A)$, and this ideal is topologically nilpotent, showing (2).

Now suppose that (2) holds. Then, by Proposition~\ref{prod-tri-prop}, $A$ is triangularizable, and hence so is $R$, as a subset of $A$, proving (1). Moreover, by the same proposition, $\rad(A)=\tnil(A)$, and if $R$ is closed, then also $\rad(R)=\tnil(R)$, by Proposition~\ref{nil-rad-prop}. Hence, in this situation, $$\rad(R) = \tnil (R) = R \cap \tnil (A) = R\cap \rad(A),$$ proving the final claim.
\end{proof}

The next example shows that for a closed triangularizable subalgebra $R$ of $\End_k(V)$ it is not necessarily the case that $R/\rad(R)\cong k^{\Omega}$ (as topological $k$-algebras) for some set $\Omega$.

\begin{example} \label{non-iso-eg}
Let $k$ be a countable field, let $\Omega = \Z^+ \cup \{\infty\}$, and let $V$ be a $k$-vector space with basis $\B = \{v_i \mid i \in \Omega\}$. Then $\B$ is well-ordered by the relation $\preceq$, where $v_i \preceq v_j$ if and only if either $i,j \in \Z^+$ and $i\leq j$, or $j = \infty$. For each $i \in \Z^+$ define $E_i \in \End_k(V)$ by
$$E_i(v_j) = 
\left\{ \begin{array}{ll}
v_i & \text{if } \, j \in \{i, \infty\}\\
0 & \text{otherwise}
\end{array}\right.,$$ and extend linearly to all of $V$. Then $\{E_i \mid i \in \Z^+\}$ is a set of orthogonal idempotents, which is clearly triangular with respect to $(\B, \preceq)$. It follows that the $k$-subalgebra $R$ of $\End_k(V)$ generated by this set is also triangular with respect to $(\B, \preceq)$. 

Every element of $R$ can be expressed in the form $a\cdot 1 + \sum_{i=1}^n a_iE_i$, for some $n \in \Z^+$ and $a, a_i \in k$. Notice that $$\bigg(a\cdot 1 + \sum_{i=1}^n a_iE_i\bigg)(v_{\infty}) = av_{\infty} + \sum_{i=1}^n a_iv_i,$$ and so each element of $R$ is completely determined by its action on $v_{\infty}$. This implies that $R$ is discrete in the topology induced on it by the function topology on $\End_k(V)$. The above computation also shows that if $a \neq 0$, then $(a\cdot 1 + \sum_{i=1}^n a_iE_i)^m(v_{\infty}) \neq 0$ for all $m \in \Z^+$. Moreover, if $a=0$ but $a_n \neq 0$, then $$\bigg(a\cdot 1 + \sum_{i=1}^n a_iE_i\bigg)^m(v_n) = a_n^mv_n \neq 0$$ for all $m \in \Z^+$. Thus $R$ has no nonzero topologically nilpotent elements, and hence $\rad(R) = 0$, by Proposition~\ref{nil-rad-prop}.

Define $\phi : R \to k^{\Omega}$ as in Proposition~\ref{hom-prop}. Specifically, $\phi(a\cdot 1 + \sum_{i=1}^n a_iE_i) = (b_i)_{i \in \Omega}$, where $b_i = a_i +a$ for $1 \leq i \leq n$ and $b_i=a$ for all $j > n$. Thus $\phi$ is an injective $k$-algebra homomorphism (which is also continuous, by Proposition~\ref{hom-prop}), with $\phi(R) = \langle k^{(\Z^+)} \cup \{1\} \rangle$, where $k^{(\Z^+)}$ is the direct sum of copies of $k$ indexed by the elements of $\Z^+$ (with each element of this subring of $k^{\Omega}$ understood to have $0$ in the $\infty$ coordinate). Therefore $$R/\rad(R) \cong R \cong \langle k^{(\Z^+)} \cup \{1\} \rangle$$ as $k$-algebras. In particular, $R/\rad(R) \not\cong k^{\Delta}$ for any set $\Delta$, since such a $\Delta$ would need to be infinite, making $k^{\Delta}$ uncountable, in contrast to $R$ (given that $k$ was assumed to be countable). 

Next, let us show that $R$ is closed. Let $\cl(R)$ denote the closure of $R$ in the function topology on $\End_k(V)$, and let $T \in \cl(R)$. Then there exist $n\in \Z^+$ and $a, a_1, \dots, a_n \in k$ such that $a\cdot 1 + \sum_{i =1}^n a_iE_i \in R$ agrees with $T$ on $v_{\infty}$. Since each element of $R$ is completely determined by its action on $v_{\infty}$, and since every neighborhood of $T$ must contain an element of $R$, it follows that $T = a\cdot 1 + \sum_{i =1}^n a_iE_i$, which belongs to $R$, and hence $R=\cl(R)$.

We also observe that $\phi$ is not open. Since $R$ is discrete, $\{E_1\} \subseteq R$ is open. On the other hand the intersection of any open subset of $k^{\Omega}$ with $\langle k^{(\Z^+)} \cup \{1\} \rangle$ is infinite, and hence $\phi(\{E_1\})$ cannot be open in the induced topology on $\langle k^{(\Z^+)} \cup \{1\} \rangle$.

Finally, let $A$ be the subalgebra of $\End_k(V)$ consisting of all the transformations triangular with respect to $(\B, \preceq)$, and extend $\phi$ to a map $\phi' : A \to k^{\Omega}$ (again, defined as in Proposition~\ref{hom-prop}). Then $\phi'$ is open and surjective, by Proposition~\ref{hom-prop}(3), but $\phi'$ is not closed, since it sends the closed set $R$ to $\langle k^{(\Z^+)} \cup \{1\} \rangle$, whose closure in $k^{\Omega}$ is $k^{\Omega}$. \hfill $\Box$
\end{example}

\section{Algebraically Closed Fields}

Next we wish to add another condition to Theorem~\ref{main-theorem}, in the case where the field is algebraically closed. This will require defining more topological terms and recalling an earlier result.

\begin{definition}
Let $k$ be a field and $R$ a topological $k$-algebra. Then the ring $R$ is called \emph{pseudocompact} if the topology on $R$ is complete, is Hausdorff, and has a basis of neighborhoods of $\, 0$ consisting of ideals $I$ such that $R/I$ has finite length both on the left and on the right $\, ($i.e., $R/I$ is a two-sided artinian ring$)$. We also say that $R$ is $k$-\emph{pseudocompact} if it is pseudocompact and every open ideal of $R$ has finite $k$-codimension.
\end{definition}

\begin{proposition}[Part of Proposition 3.15 in~\cite{IMR}]\label{WA-fin-dim}
Let $k$ be a field and $R$ a topological $k$-algebra. Then the following are equivalent.
\begin{enumerate}
\item[$(1)$]  $R$ is $k$-pseudocompact and $\, \rad(R)=\{0\}$.
\item[$(2)$]  $R \cong \prod_{\alpha \in \Omega} \M_{n_{\alpha}}(D_{\alpha})$ as topological $k$-algebras, for some set $\, \Omega$, $n_{\alpha} \in \Z^+$, and $D_{\alpha}$ finite-dimensional division $k$-algebras, where $\, \prod_{\alpha \in \Omega} \M_{n_{\alpha}}(D_{\alpha})$ is given the product topology with each $\, \M_{n_{\alpha}}(D_{\alpha})$ discrete.
\end{enumerate}
\end{proposition}

We are now ready for the extension of Theorem~\ref{main-theorem} in the case where the field is algebraically closed.

\begin{corollary} \label{alg-closed-cor}
Let $k$ be an algebraically closed field, $\, V$ a nonzero $k$-vector space, and $R$ a $k$-subalgebra of $\, \End_k(V)$. Then the following are equivalent.
\begin{enumerate}
\item[$(1)$] $R$ is triangularizable.
\item[$(2)$] $R$ is contained in a $k$-subalgebra $A$ of $\, \End_k(V)$ such that $A/\rad(A) \cong k^{\Omega}$ as topological $k$-algebras for some set $\, \Omega$, and $\, \rad (A)$ is topologically nilpotent.
\item[$(3)$] $R$ is contained in a $k$-subalgebra $A$ of $\, \End_k(V)$ such that $A/\rad(A)$ is $k$-pseudocompact and commutative, and $\, \rad (A)$ is topologically nilpotent.
\end{enumerate}
\end{corollary}

\begin{proof}
The equivalence of (1) and (2) follows from Theorem~\ref{main-theorem}. Also, by Proposition~\ref{WA-fin-dim}, if $A/\rad(A) \cong k^{\Omega}$ as topological $k$-algebras, then $A/\rad(A)$ is $k$-pseudocompact, and hence (2) implies (3).

Finally, let us suppose that (3) holds and prove (2). By Proposition~\ref{WA-fin-dim}, $A/\rad(A)$ being $k$-pseudocompact implies that $A/\rad(A) \cong \prod_{\alpha \in \Omega} \M_{n_{\alpha}}(D_{\alpha})$ as topological $k$-algebras, where each $D_{\alpha}$ is a finite-dimensional division algebra over $k$. Since we are assuming that $A/\rad(A)$ is commutative, each $n_{\alpha} = 1$, and each $D_{\alpha}$ must be a finite-dimensional field extension of $k$. Since $k$ is algebraically closed, each $D_{\alpha}$ must be isomorphic to $k$, and hence $A/\rad(A) \cong k^{\Omega}$ as topological $k$-algebras, from which (2) follows.
\end{proof}

We can now give a generalization of the well-known result, proved by McCoy~\cite{McCoy}, that a subalgebra of a matrix ring over an algebraically closed field is triangularizable if and only if the subalgebra is commutative modulo its Jacobson radical.

\begin{corollary} \label{McCoy-cor}
Let $k$ be a field, $n \in \Z^+$, and $R$ a $k$-subalgebra of $\, \M_n(k)$. Then the following are equivalent.
\begin{enumerate}
\item[$(1)$] $R$ is triangularizable.
\item[$(2)$] $R/\rad (R) \cong k^m$ as $k$-algebras, for some $m \in \Z^+$.
\end{enumerate}
If $k$ is algebraically closed, then these are also equivalent to the following.
\begin{enumerate}
\item[$(3)$] $R/\rad (R)$ is commutative.
\end{enumerate}
Moreover, if $\, (2)$ holds, then $m\leq n$.
\end{corollary}

\begin{proof}
Suppose that (1) holds. Since $\M_n(k)$ is finite-dimensional as a $k$-algebra, it and $R$ are discrete in the function topology, and in particular, $R$ is closed. Hence, by Theorem~\ref{main-theorem}, $R/\rad (R)$ is isomorphic to a subalgebra of $k^r$, for some $r \in \Z^+$. It is well-known that every subalgebra of $k^r$ is of the form $k^m$ for some $m \in \Z^+$. (See~\cite[Corollary 4.8]{IMR} for a proof of a more general version of this statement.) Therefore $R/\rad (R) \cong k^m$ as $k$-algebras, proving (2).

Next suppose that (2) holds. Since $R$ is finite-dimensional, and hence left artinian, $\rad(R)$ is (topologically) nilpotent (see, e.g.,~\cite[Proposition IX.2.13]{Hungerford} or~\cite[Theorem 4.12]{Lam}). Moreover, since $R$ is discrete in the function topology and $k^m$ is discrete in the product topology, $R/\rad (R) \cong k^m$ is a topological isomorphism. Therefore $R$ is triangularizable, by Theorem~\ref{main-theorem}, showing (1).

Clearly, (2) implies (3). Let us next assume that $k$ is algebraically closed and that (3) holds, and show that (2) must also hold. Since $R/\rad (R)$ is left artinian, and since it has zero Jacobson radical, $R/\rad (R)$ is semisimple (see, e.g.,~\cite[Theorem 4.14]{Lam}). Thus, $R/\rad (R)$ is isomorphic as a $k$-algebra to a finite direct product of matrix rings over division $k$-algebras (see, e.g.,~\cite[Theorem IX.5.4]{Hungerford}). Since $R/\rad (R)$ is commutative, each of these matrix rings must be $1 \times 1$, and each of the division $k$-algebras must be a field extension of $k$. Since $R/\rad (R)$ is a finite-dimensional $k$-algebra, and $k$ is assumed to be algebraically closed, each of these fields must be isomorphic to $k$. Hence $R/\rad(R) \cong k^m$ as $k$-algebras, for some $m \in \Z^+$.

Finally, if $(2)$ holds, then, as mentioned above, $\rad(R) \subseteq \tnil(R)$. Since $k^m$ has no nonzero nilpotent elements,  it follows that $\rad(R) = \tnil(R)$. Since $R$ is triangularizable, by the equivalence of $(1)$ and $(2)$, Proposition~\ref{hom-prop} implies that $R/\rad(R)$ is isomorphic to a subalgebra of $k^n$. Thus $$m = \dim_k(k^m) = \dim_k(R/\rad(R)) \leq \dim_k(k^n) = n,$$ which shows that $m\leq n$.
\end{proof}

\section*{Acknowledgements}

I am grateful to George Bergman for a very enlightening conversation about this material, and for numerous comments on an earlier draft of this paper, which have led to significant improvements. I also would like to thank Greg Oman for a pointer to the literature, Manuel Reyes for a helpful example, and the referee for a careful reading of the manuscript.

\medskip

\noindent Department of Mathematics, University of Colorado, Colorado Springs, CO, 80918, USA 

\noindent \emph{Email:} zmesyan@uccs.edu 

\end{document}